\documentclass[smallextended,referee,envcountsect]{svjour3}
\usepackage{eurosym}
\usepackage[utf8]{inputenc}
\usepackage[fleqn]{amsmath}
\usepackage{amsmath, amssymb, amstext,amsfonts}
\usepackage{graphicx}
\usepackage{color}
\usepackage{textcomp}
\usepackage{graphics}

\newtheorem{assumption}{Assumption}[section]

\newcommand{\no}{\noindent}
\newcommand{\be}{\begin{equation}}
\newcommand{\ee}{\end{equation}}

\newcommand{\bea}{\begin{eqnarray}}
\newcommand{\bes}{\begin{subequations}}
\newcommand{\ees}{\end{subequations}}
\newcommand{\bgt}{\begin{gather}}
\newcommand{\egt}{\begin{gather}}
\newcommand{\eea}{\end{eqnarray}}
\newcommand{\beaa}{\begin{eqnarray*}}
\newcommand{\eeaa}{\end{eqnarray*}}

\newcommand{\NN}{{\mathbb N}}
\newcommand{\EE}{{\mathbb E}}
\newcommand{\RR}{{\mathbb R}}
\renewcommand{\cal}{\mathcal}

\def \d{\delta}
\def \g {\gamma}
\def \l {\lambda}
\def \m {\mu}
\def \p {\phi}
\def \t {\tau}
\def \ps {\psi}

\def \D{\Delta}

\def\Dc{{\cal D}}
\def\Fc{{\cal F}}
\def\Hc{{\cal H}}
\def\Lc{{\cal L}}

\def\Pc{{\cal P}}
\def\Tc{{\cal T}}

\def \F{\mathbb{F}}
\def \E{\mathbb{E}}
\def \Eb{\mathbb{E}}
\def \H{\mathbb{H}}
\def\P{\mathbb{P}}

\def \Rb{\mathbb{R}}
\def \R{\mathbb{R}}

\def \bu {\boldsymbol{\mu}}

\def\reff#1{{\rm(\ref{#1})}}

\def \t {\tau}

\newcommand{\rmi}{{\rm (i)$\>\>$}}
\newcommand{\rmii}{{\rm (ii)$\>\>$}}
\newcommand{\rmiii}{{\rm (iii)$\>\>$}}

\begin{document}

\title{Some Results on Skorokhod Embedding and Robust Hedging with Local Time}

\author{Julien Claisse \and Gaoyue Guo \and  Pierre Henry-Labord\`ere}

\institute{Julien Claisse,  Corresponding author \at
             \'Ecole Polytechnique \\
              Palaiseau, France\\
              claisse@cmap.polytechnique.fr 
           \and
           Gaoyue Guo  \at
              University of Oxford \\
              Oxford, United Kingdom\\
              guo.gaoyue@gmail.com
           \and 
           Pierre Henry-Labord\`ere \at
           	  Soci\'et\'e G\'en\'erale \\
           	  Paris, France\\
           	  pierre.henry-labordere@sgcib.com
}

\date{Received: date / Accepted: date}

\maketitle

\begin{abstract}
	In this paper, we provide some results on Skorokhod embedding with local time and its applications to the robust hedging problem in finance. 
	First we investigate the robust hedging of options depending on the local time by using the recently introduced stochastic control approach, in order to identify the optimal hedging strategies, as well as the market models that realize the extremal no-arbitrage prices. 
	As a by-product, the optimality of Vallois' Skorokhod embeddings is recovered. 
	In addition, under appropriate conditions, we derive a new solution to the two-marginal Skorokhod embedding as a generalization of the Vallois solution. It turns out from our analysis that one needs to relax the monotonicity assumption on the embedding functions in order to embed a larger class of marginal distributions.
	Finally, in a full-marginal setting where the stopping times given by Vallois are well-ordered, we construct a remarkable Markov martingale which provides a new example of fake Brownian motion.
\end{abstract}



\section{Introduction}

\no The Skorokhod embedding problem (SEP for short) consists in choosing a stopping time in order to represent a given probability on the real line as the distribution of a stopped Brownian motion. First formulated and solved by Skorokhod~\cite{skorokhod-61}, this problem has given rise to important literature and a large number of solutions have been provided. We refer the reader to the survey paper by Obl{\'o}j~\cite{obloj-04} for a detailed description of the known solutions. 
Among them, the solution provided by Vallois~\cite{vallois-83} is based on the local time (at zero) of the Brownian motion.
As proved later by Vallois~\cite{vallois-92}, it has the property to maximize the expectation of any convex function of the local time among all solutions to the SEP. Similarly, many solutions to the SEP satisfy such an optimality property. This feature has given rise to important applications in finance as it allows to solve the so-called robust hedging problem which we describe below.

\no The classical pricing paradigm of contingent claims consists in postulating first a model, i.e., a  risk-neutral measure, under which forward prices are required to be martingales according to the no-arbitrage framework. 
Then the price of any European derivative is obtained as the expectation of its discounted payoff under this measure.   Additionally, the model may be required to be calibrated to the market prices of liquid options such as call options that are available for hedging the exotic derivative under consideration.
This could lead to a wide range of prices when evaluated using different models calibrated to the same market data. To account for the model uncertainty, it is natural to consider simultaneously a family of (non-dominated) market models. Then the seller (resp. buyer) aims to construct a portfolio to super-replicate (resp. sub-replicate) the derivative under any market scenario by trading dynamically in the underlying assets and statically in a range of Vanilla options. This lead to an interval of no-arbitrage prices whose bounds are given by  the minimal super-replication and the maximal sub-replication prices. The robust hedging problem is to compute these bounds as well as the corresponding trading strategies.

\no  We consider the classical framework where all European call options having the same maturity as the exotic derivative are available for trading. 
As observed by Breeden and Litzenberger~\cite{breeden-litzenberger-78}, the marginal distribution of the underlying price process at maturity is uniquely determined by the market prices of these call options. In this setting, the robust hedging problem is classically approached by means of the SEP. This approach relies on the fact that every continuous martingale can be considered as a time-changed Brownian motion. Thus, for payoffs invariant under time change, the problem can be formulated in terms of finding a solution to the SEP which optimizes the criterion given by the payoff. This approach was initiated by Hobson who considered the robust superhedging problem for lookback options in his seminal paper~\cite{hobson-98}. 
Since then, the SEP has received substantial attention from the mathematical finance community and this approach was subsequently exploited in Brown, Hobson and Rogers~\cite{brown-hobson-rogers-01-2} for barrier options, in Cox, Hobson and Obl{\'o}j~\cite{cox-hobson-obloj-08} for options on local time, in Cox and Obl{\'o}j~\cite{cox-obloj-11} for double-barrier options and in Cox and Wang~\cite{cox-wang-13} for options on variance. 
One of the key steps in the SEP approach is to guess the form of the optimal hedging strategies from a well-chosen pathwise inequality.

\no Recently, a new approach to study the robust hedging problem was developed by Galichon, Henry-Labord\`ere and Touzi~\cite{galichon-henry-touzi-14}. It is based on a dual representation of the robust hedging problem, which can be addressed by means of the stochastic control theory. It appears that the stochastic control approach is remarkably devised to provide candidates for the optimal hedging strategies. Once postulated, the hedging inequality can be verified independently. %Thus, the stochastic control approach comes as a complement to the SEP to study the robust hedging problem. 
This is illustrated by the study of Henry-Labord\`ere et al.~\cite{henry-al-15}, where they solve the robust hedging problem for lookback options when a finite number of marginals are known. To the best of our knowledge, 
it is the first paper to address the multi-marginal problem, at the exception of Brown, Hobson and Rogers~\cite{brown-hobson-rogers-01-1} and Hobson and Neuberger~\cite{hobson-neuberger-12} who considered the two-marginal case. In addition, it led to the first (nontrivial) solution of the multi-marginal SEP, which can be seen as a generalization of the Az\'ema-Yor solution, see Obl{\'o}j and Spoida~\cite{obloj-spoida-13}.

\no In this paper, we aim to collect a number of new results regarding Skorokhod embedding with local time and its applications. First we are concerned with the robust hedging problem for options written on the local time of the underlying price process. Such derivatives appear naturally in finance when considering payoffs depending on the portfolio value of an at-the-money call option delta hedged with the naive strategy holding one unit of the risky asset if in the money, else nothing, as expressed mathematically by It\^o-Tanaka's formula.
By using the stochastic control approach, we first recover the results  on the robust superhedging problem obtained by Cox, Hobson and Obl{\'o}j~\cite{cox-hobson-obloj-08}, i.e., we identify optimal superhedging strategies and the upperbound of the no-arbitrage interval. Then we derive the corresponding results for the robust subhedging problem. The last result is new to the literature.  

\no In addition, we provide a new solution to the two-marginal SEP as a generalization of the Vallois solution. To this end, we have to make rather strong assumptions on the marginals that we aim to embed. However it is remarkable that these assumptions are to a certain extent necessary to derive a solution without relaxing the monotonicity assumption on the embedding functions. To the best of our knowledge, this is one of the first solution to the multi-marginal SEP to appear in the literature. In addition to the purely theoretical interest of this result, it is also a first step toward a solution to the robust hedging problem in the two-marginal setting, i.e., when the investor can trade on Vanilla option with an intermediate maturity. 

\no Finally, we consider a special full-marginal setting when the stopping times given by Vallois are well-ordered. In the spirit of Madan and Yor~\cite{madan-yor-02}, we construct a remarkable Markov martingale via the family of Vallois' embeddings and compute its generator. In particular, it provides a new example of fake Brownian motion. From a financial viewpoint, our result characterizes the arbitrage-free model calibrated to the full implied volatility surface, which attains the upper bound of the no-arbitrage interval when the investor can trade in Vanilla options maturing at any time.

\no The paper is organized as follows. We briefly introduce in Section 2 the framework of robust hedging of exotic derivatives and its relation with the martingale optimal transport problem and the SEP. In Section 3, using the stochastic control approach, we provide explicit formulas for the bounds of the no-arbitrage interval and the optimal hedging strategies for the robust hedging problem. Then we introduce our new solution to the two-marginal SEP in Section 4. We illustrate this result by studying a numerical example.  Finally, we consider in Section 5 the full-marginal setting and construct our new example of fake Brownian motion.

\section{Formulation of the Robust Hedging Problem}
\label{sec:robust}

\subsection{Modeling the Model Uncertainty}

\no We consider a financial market consisting of one risky asset, which may be traded at any time $0\le t \le T$, where $T$ denotes some fixed maturity. We pursue a robust approach and do not specify the dynamics of the underlying price process. Namely, given an initial value  $X_0\in\R$, we introduce the set of continuous paths $\Omega:= \{\omega\in C([0,T], \Rb): \omega(0)=X_0\}$ as the canonical space equipped with the uniform norm $\|\omega\|_\infty :=
\sup_{0\le t\le T}|\omega(t)|$. Let $X=(X_t)_{0\leq t\le T}$ be the canonical process
and $\F= (\Fc_t)_{0\le t\le T}$ be the natural filtration, i.e., $X_t(\omega):=\omega(t)$ and $\Fc_t:=\sigma(X_s, s\le t)$. In this setting, $X$
stands for the underlying price process with initial value $X_0$.
In order to account for model uncertainty, we introduce the set $\Pc$ of all probability measures $\P$ on $(\Omega,\Fc_T)$ such that $X$ is a $\P-$martingale. The restriction to martingale measures is motivated by the classical no-arbitrage framework in mathematical finance. For the sake of generality, we do not restrict to $X_t \in \RR_+$ but consider the general case $X_t \in \RR$.

\no In addition, all call options with maturity $T$ are assumed to be available for trading.
A model $\P\in\Pc$ is said to be calibrated to the market if it satisfies $\E^\P[(X_{T}-K)^+]=c(K)$ for all $K\in \R$, 
where $c(K)$ denotes the market price of a $T-$call option with strike $K$.
For such a model, as observed by Breeden and Litzenberger~\cite{breeden-litzenberger-78}, it follows  by direct differentiation that
 \begin{eqnarray*}
 \P(X_{T}>K)=-c'(K+)=:\mu{\left(]K,\infty[\right)}. 
 \end{eqnarray*}
Hence the marginal distribution of $X_{T}$ is uniquely specified by the market prices. Let $\Pc^{\mu}$ be the set of calibrated market models, i.e.,
 \begin{eqnarray*} 
 \Pc^{\mu}
 &:=&
 \left\{\P\in\Pc:~ X_{T}~\stackrel{\P}{\sim}~\mu\right\}.
 \end{eqnarray*}
 Clearly, $\Pc^{\mu}\neq \emptyset$ if and only if $\mu$ is centered at $X_0$ or, equivalently,
\beaa
  \int_{\R} |x| \,d\mu(x)<\infty &~\mbox{and}~&
\int_{\R} x \,d\mu(x)=X_0.
\eeaa

\subsection{Semi-Static Hedging Portfolios}
\label{sect:semistatic}

We denote by $\H^0$ the collection of all $\F-$predictable processes and, for every $\P\in\Pc$,  
 \begin{eqnarray*}
 \H^2(\P)
 &:=&
 \Big\{\Delta=(\Delta_t)_{0\le t\le T}\in\H^0:\, \int_0^T |\Delta_t|^2 \,d\langle X\rangle_t < \infty,
                     ~\P-\mbox{a.s.}
 \Big\}.
 \end{eqnarray*}
A dynamic trading strategy is defined by a process $\Delta \in \H^2:=\cap_{\P\in\Pc} \H^2(\P)$, 
where $\Delta_t$ corresponds to the number of shares of the underlying asset held by the investor at time $t$. 
Under the self-financing condition, the portfolio value process of initial wealth $Y_0$ induced by a dynamic trading strategy $\D$ is given by
\footnote{
Both the quadratic variation and the stochastic integral depend \textit{a priori} on the probability measure under consideration. However, under the Continuum Hypothesis, it follows by Nutz~\cite{nutz-12} that they can be universally defined.
}
 \begin{eqnarray*}\label{XH}
 Y_t^\Delta
 &:=&
 Y_0 + \int_0^t \Delta_s \,dX_s, \quad \mbox{for all } t\in[0,T],~
 \P-\mbox{a.s. for all } \P\in\Pc.
 \end{eqnarray*}

 \no In addition to the dynamic trading on the underlying security, we assume that the investor can take static positions in $T-$call options for all strikes. Consequently, up to integrability, the European derivative defined by the payoff $H(X_T)$, which can be statically replicated by $T$-call options in view of the celebrated Carr-Madan formula, has an unambiguous market price
 \begin{eqnarray*}
 \mu(H) 
 &:=& 
 \int_\R H(x) \, d\mu (x).
 \end{eqnarray*}
 The set of Vanilla payoffs which may be used by the trader has naturally the following form
  \begin{eqnarray*}\label{Lambdan}
  L^1(\mu) 
  &:=& 
  \Big\{H:\R\to\R~ measurable~ s.t.~  \m\big(|H|\big)<\infty \Big\}.
  \end{eqnarray*}
  A pair $(\Delta, H)\in \H^2\times L^1(\mu)$ is called a semi-static hedging strategy, and induces the final value of the self-financing portfolio:
 \begin{eqnarray*}\label{barX}
 {Y}^{\Delta,H}_T
 &:=&
 Y^\Delta_T - \mu(H) + H(X_T),\quad \P -\mbox{a.s. for all } \P\in\Pc,
 \end{eqnarray*}
indicating that the investor has the possibility of buying at initial time any derivative with payoff $H(X_T)$ for the price $\mu(H)$. 
 
\subsection{Robust Hedging and Martingale Optimal Transport}
\label{sec:robusthedging}

Given a derivative of payoff $\xi=\xi(X)$ $\Fc_T$-measurable, we consider  the corresponding problem of robust (semi-static) hedging. The investor can trade as discussed in the previous section. However we need to impose a further admissibility condition to rule out doubling strategies. 
Let $\overline{\Hc}^{\mu}$ (resp. $\underline{\Hc}^{\mu}$) consist of all processes $\Delta \in \H^2$ whose induced portfolio value process $Y^\Delta$ is a $\P-$supermartingale (resp. $\P-$submartingale) for all $\P\in\Pc^{\mu}$. 
The robust superhedging and subhedging costs are then defined by
\beaa
 U^{\mu}(\xi)
 &:=&
 \inf\left\{ Y_0: \exists (\Delta, H)\in \overline{\Hc}^{\mu}\times L^1(\mu) \mbox{ s.t. } 
                  {Y}^{\Delta, H}_T\ge \xi,~
                  \P -\mbox{a.s. }\forall\, \P\in\Pc
     \right\},\\
 D^{\mu}(\xi)
 &:=&
 \sup\left\{ Y_0:  \exists (\Delta, H)\in \underline{\Hc}^{\mu}\times L^1(\mu) \mbox{ s.t. } 
                  {Y}^{\Delta, H}_T\le \xi,~
                  \P -\mbox{a.s. }\forall\,  \P\in\Pc
     \right\}.
\eeaa
Selling $\xi$ at a price higher than $U^{\mu}(\xi)$ --- or buying it at a price lower than $D^{\mu}(\xi)$ --- the trader could set up a portfolio with a negative initial cost and a non-negative payoff under any market scenario leading to a strong (model-independent) arbitrage opportunity. 

\no By taking expectation in the hedging inequalities under $\P\in\Pc^{\mu}$,
we obtain the usual pricing--hedging inequalities:
\beaa
  U^{\mu}(\xi)  \geq \sup_{\P\in\Pc^{\mu}} \EE^{\P}[\xi] =: P^{\mu}(\xi) &~\mbox{and}~& D^{\mu}(\xi)  \leq \inf_{\P\in\Pc^{\mu}} \EE^{\P}[\xi] =: I^{\mu}(\xi), \label{weakdual}
\eeaa
where $P^{\mu}(\xi)$ and $I^{\mu}(\xi)$ are continuous-time martingale optimal transport problems. They consist in maximizing or minimizing the criterion defined by the payoff so as to transport the Dirac measure at $X_0$ to the given distribution $\mu$ by means of a continuous-time process restricted to be a martingale. 

\no The study of martingale optimal transportation was recently initiated by Beiglb\"ock, Henry-Labord\`ere and Penkner~\cite{beigblock-al-13} in discrete-time and by Galichon, Henry-Labord\`ere and Touzi~\cite{galichon-henry-touzi-14} in continuous-time. By analogy with the classical optimal transportation theory, one expects to establish a sort of Kantorovitch duality and to characterize the optimizers for both the primal and dual problems. The dual formulation has a natural financial interpretation in terms of robust hedging, which explains the keen interest of the mathematical finance community in martingale optimal transport. When the payoff is invariant under time-change, the SEP and the stochastic control approach turn out to be powerful tools to derive the duality and compute explicitly the optimizers as illustrated in this paper. 
For duality results with more general payoffs, we refer to the recent studies by Dolinsky and Soner~\cite{dolinsky-soner-14}, Hou and Obl{\'o}j~\cite{hou-obloj-15} and Guo, Tan and Touzi~\cite{guo-tan-touzi-15-2}.

\subsection{Robust Hedging of Options on Local Time}

\no In this paper, we focus on the robust hedging problem of options whose payoff is given by
\beaa
 \xi &:=& F(L_T) \mbox{ with } F: \R_+\longrightarrow\R,
\eeaa
 where $L=(L_t)_{0\le t\le T}$ is the local time of $X$ at $X_0$. Below, under appropriate conditions, we will exhibit the optimizers for both the robust hedging and the martingale optimal transport problems, and show further that there is no duality gap, i.e., $U^{\mu}(F(L_T)) = P^{\mu}(F(L_T))$ and $D^{\mu}(F(L_T)) = I^{\mu}(F(L_T))$. 
 
 \no The payoff $F(L_T)$ can be interpreted as a payoff  depending on the portfolio value at maturity $T$ of an at-the-money call option delta hedged with the naive strategy holding one unit of the risky asset if in the money, else nothing, mathematically expressed by It\^o-Tanaka's formula:
\footnote{
In view of the pathwise construction of stochastic integrals in Nutz~\cite{nutz-12}, It\^o-Tanaka's formula implies that the local time can also be universally defined.  
}
 \beaa
   \frac{1}{2} L_T&=&(X_T-X_0)^+ - \int_0^T {\mathbf 1}_{\{X_t>X_0\}} dX_t.
 \eeaa

\no Since the local time is invariant under time-change,
\footnote{
Namely, given a family of stopping times $(\tau_t)_{t\geq 0}$ such that $t\mapsto \tau_t$ is continuous and increasing, we have $(L_{\tau_t})_{t\geq 0}=(\tilde{L}_t)_{t\geq 0}$ where $(\tilde{L}_t)_{t\geq 0}$ denote the local time at $X_0$ of the process $(X_{\tau_t})_{t\geq 0}$.
}
the martingale optimal transport problem can be formulated as an optimal stopping problem.  Indeed, it follows (formally) from the Dambis-Dubins-Schwartz theorem that
\bea\label{SEP}
 P^{\mu}(F(L_T)) = \sup_{\t\in\Tc^{\mu}}
 \E\left[F(L^B_{\tau})\right] &~\mbox{and}~&
 I^{\mu}(F(L_T)) = \inf_{\t\in\Tc^{\mu}}
 \E\left[F(L^B_{\tau})\right],
\eea
 where $(L^B_t)_{t\ge 0}$ is the local time at zero of a Brownian motion $(B_t)_{t\ge 0}$ and $\Tc^{\mu}$ is the collection of solutions to the SEP, i.e., stopping times $\tau$ such that
 \beaa
 B^{\tau} := (B_{t\wedge\t})_{t\geq0} \mbox{ is uniformly integrable} &\mbox{and}& B_{\tau}\sim\mu. 
 \eeaa
 See Galichon, Henry-Labord\`ere and Touzi~\cite{galichon-henry-touzi-14} and Guo, Tan and Touzi~\cite{guo-tan-touzi-15} for more details. 
Here, the formulation \reff{SEP} is directly searching for a solution to the SEP which maximizes or minimizes the criterion defined by the payoff.
 
\no It is well known that, if $F$ is a convex (or concave) function, the optimal solutions are of the form
\begin{eqnarray*}
 \t & :=& \inf \Big\{t>0:~ \ B_t\notin\big]\p_-(L^B_t), \p_+(L^B_t)\big[\Big\},
\end{eqnarray*}
for some monotone functions $\p_{\pm}:\R_+\to\R_{\pm}$.  
This result was first obtained in Vallois~\cite{vallois-92}, where he gives explicit constructions for the functions $\p_{\pm}$. It was then recovered by Cox, Hobson and Obl{\'o}j~\cite{cox-hobson-obloj-08} from a well-chosen pathwise inequality. 
More recently,
it was derived by Beiglb\"ock, Cox and Huesmann~\cite{beiglboeck-cox-huesmann-13} as a consequence of their monotonicity principle, which characterizes optimal solutions to the SEP by means of their geometrical support. 
%Indeed it allows to characterize the optimal solution to~\eqref{SEP} by means of its geometrical support. 
However, the explicit computation of $(\p_+,\p_-)$ is not provided by this approach. See also Guo, Tan and Touzi~\cite{guo-tan-touzi-16} on the monotonicity principle.

\section{Solution of the Robust Hedging Problem}\label{sec:one-marginal}

\no Using the stochastic control approach, we reproduce in this section the results for the robust superhedging problem --- optimizers and duality --- obtained in Cox, Hobson and Obl{\'o}j~\cite{cox-hobson-obloj-08}. 
In addition, we provide the corresponding results for the robust subhedging problem. Throughout this section, we take $X_0=0$ for the sake of clarity and we work under the following assumption on the function $F$ and on the marginal $\mu$. In particular, in contrast with~\cite{cox-hobson-obloj-08}, we do not need to assume that $F$ is nondecreasing. 

\begin{assumption}\label{hyp:F}
 $F: \R_+\to\R$ is a Lipschitz convex function.
\end{assumption}

\begin{assumption}\label{hyp:mu} 
 $\mu$ is a centered probability distribution without mass at zero.
\end{assumption}

\subsection{Robust Superhedging Problem}\label{sec:onemarginal-sup}

In this section, under Assumptions \ref{hyp:F} and \ref{hyp:mu}, we provide the optimal hedging strategy for $U^{\mu}(F(L_T))$ as well as the optimal measure for $P^{\mu}(F(L_T))$ and we show that there is no duality gap, i.e., $P^{\mu}(F(L_T))=U^{\mu}(F(L_T))$. The key idea is that for any suitable pair of monotone functions $(\p_+,\p_-)$, we may construct a super-replication strategy $(\Delta, H)=(\Delta^{\phi_{\pm}}, H^{\phi_{\pm}})$. The optimality then results from taking the pair of functions given by Vallois that embeds the distribution $\mu$.

\begin{assumption}\label{hyp:onemarginal-phi}
$\p_+:]0,\infty[~ \longrightarrow~ ]0,\infty[$ (resp. $\p_-:]0,\infty[~ \longrightarrow~ ]-\infty,0[$) is right-continuous and nondecreasing (resp. nonincreasing) such that $\g(0+)=0$  and $\g(\infty)=\infty$ where, for all $l>0$,
 \begin{eqnarray*}
  \g(l)&:=&\frac{1}{2} \int_0^l\left(\frac{1}{\p_+(m)}-\frac{1}{\p_-(m)}\right)dm.
 \end{eqnarray*}
\end{assumption}

\no Let us denote by $\psi_\pm$ the right-continuous inverses of $\p_\pm$. We also define the functions $A_\pm:\R_+\to\R$  via $(\p_+,\p_-)$ by
\begin{eqnarray}
 A_\pm(l) & := & A_\pm(0) + \int_0^{l} {\frac{dz}{\p_\pm(z)}} e^{\g(z)}\int_z^{\infty} {e^{-\g(m)} \,F''(dm)}, \label{def:A}\\
 A_\pm(0) &:=& \pm F'(0) \pm \int_0^{\infty} {e^{-\g(m)} \,F''(dm)}. \label{def:H'(0)}
\end{eqnarray}
Throughout this paper, the derivatives under consideration are in the sense of distributions, and whenever possible, we pick a ``nice'' representative for such distribution. In particular, in the formula above, $F'$ and $F''$ stand for the right derivative of $F$ and the Lebesgue-Stieltjes measure relative to $F'$ respectively, which are well-defined since $F$ is a convex function.

\subsubsection{{\rm Quasi-Sure Inequality}}\label{sec:pathwise}

We start by showing a quasi-sure inequality, which is a key step in our analysis.
 It implies that, in order to construct a super-replication strategy, it suffices to consider a pair $(\p_+,\p_-)$ satisfying Assumption \ref{hyp:onemarginal-phi}. The duality and the optimality will follow once we find an optimal pair as it will be shown in the next section.

\begin{proposition}\label{prop:pathineqsur}
\no Under Assumptions \ref{hyp:F} and \ref{hyp:onemarginal-phi}, the following inequality holds
\begin{eqnarray}\label{eq:pathineqsur} 
  \int_0^T \D_t \,dX_t + H(X_T)&\geq& F(L_T),\quad \P -\mbox{a.s. for all } \P\in\Pc, \label{eq:pathineq}
\end{eqnarray}
where
\begin{align}
 \D_t & :=   A_+(L_t) \mathbf{1}_{\{X_t > 0\}} - A_-(L_t) \mathbf{1}_{\{X_t\leq 0\}}, &\text{for all } t\in [0,T], \label{eq:delta1}\\
  H(\pm x) & :=  F(0) + \int_0^{\pm x} {A_\pm\left(\psi_\pm (y)\right)\, dy}, &\text{for all } x\geq0. \label{def:H}
\end{align}
\end{proposition}

\begin{remark}
 The derivation of the semi-static strategy $(\D,H)$ is performed in Section \ref{sect:stochcontrol} by means of the stochastic control approach. 
\end{remark}

\no Before giving the proof of Proposition~\ref{prop:pathineqsur}, we show that the quasi-sure inequality yields an upper bound for $U^{\mu}(F(L_T))$.

\begin{corollary}\label{prop:upper bound}
 Under Assumptions \ref{hyp:F} and \ref{hyp:onemarginal-phi}, one has for any centered probability measure $\mu$,
 \begin{eqnarray*} 
  P^{\mu}(F(L_T))  \leq U^{\mu}(F(L_T)) \leq \mu(H).
 \end{eqnarray*}
\end{corollary}

\begin{proof}
 In view of Proposition~\ref{prop:pathineqsur}, it suffices to show that $H\in L^1(\mu)$ and $\D\in\overline{\Hc}^\mu$. As proved in Lemma~\ref{lem:technic} below, the maps $A_\pm$ are bounded. In particular, $H'$ is bounded and thus $\mu(|H|)<\infty$. In addition, $\D$ is bounded and thus $\D\in\H^2$. It remains to prove that the local martingale $(\int_0^t {\D_s \,dX_s})_{t\geq0}$ is a supermartingale. The quasi-sure inequality \eqref{eq:pathineq} implies that for all $\P\in\Pc^{\mu}$,
 \beaa
  \int_0^t {\D_s dX_s} \ge -C(L_t~+~|X_t|) \quad\mbox{for all } t\in [0,T],~ \P -\mbox{a.s.},
 \eeaa
 where $C:=\|F'\|_{\infty}\vee \|H'\|_{\infty}$. Denote $M_t:=\int_0^t {\D_s dX_s}$ and $N_t:=C(L_t + |X_t|)$. 
 Given $(\tau_n)_{n\in\NN}$ a sequence of stopping times that reduces $(M_t)_{t\geq 0}$, it follows by Fatou's Lemma that for all $s\leq t$,
 \begin{eqnarray}\label{eq:supermartingale}
  \E\left[M_t + N_t \,|\, \Fc_s\right] \leq M_s + \liminf_{n\to\infty}\E\left[ N_{\tau_n\wedge t} \,|\, \Fc_s\right].
 \end{eqnarray} 
 In addition, clearly, $(N_t)_{t\geq 0}$ is a non-negative submartingale and thus it holds
 \begin{eqnarray*}
  0\leq N_{\tau_n\wedge t} \leq \E\left[N_t \,|\, \Fc_{\tau_n\wedge t}\right].
 \end{eqnarray*}
 In particular, the sequence $(N_{\tau_n\wedge t})_{n\in\NN}$ is uniformly integrable. Hence, it follows immediately from~\eqref{eq:supermartingale} that $(M_t)_{t\geq 0}$ is a supermartingale.  \hfill\quad \qed
\end{proof}

\no The rest of the section is devoted to the proof of Proposition~\ref{prop:pathineqsur}. We start by establishing a technical lemma.

\begin{lemma}\label{lem:technic}
With the notations of Proposition \ref{prop:pathineqsur}, the maps $A_\pm$ are uniformly bounded on $\R_+$ and it holds for all $l>0$,
 \bea
  \frac{1}{2}\left(A_+(l) - A_-(l)\right) &=& F'(l) + e^{\g(l)} \int_l^{\infty} {e^{-\g(m)} \,F''(dm)}, \label{eq:Hrel1} \\
  H(\p_+(l)) - A_+(l)\p_+(l)  &=& H(\p_-(l)) - A_-(l)\p_-(l). \label{eq:Hrel2}
 \eea
\end{lemma}

\begin{proof}
\rmi Let us start by proving~\eqref{eq:Hrel1}. We observe first that 
\begin{multline*}
  \frac{1}{2}\left(A_+(l) - A_-(l)\right) 
    = \frac{1}{2}\left(A_+(0) - A_-(0)\right) \\
    + \int_0^{l} {\g'(z) e^{\g(z)} \int_z^{\infty} {e^{-\g(m)} \,F''(dm)} \,dz}.
\end{multline*}
In addition, Fubini-Tonelli's theorem yields that 
\begin{eqnarray*}
 \int_0^{l} {\g'(z) e^{\g(z)} \int_z^{l} {e^{-\g(m)} \,F''(dm)} \,dz} &=& F'(l) -  F'(0)  - \int_0^{l} {e^{-\g(m)} \,F''(dm)}.
\end{eqnarray*}
The desired result follows immediately. 

\no\rmii Let us show next that $A_+$ is bounded. Clearly, 
\beaa
  A_+(0) \leq A_+(l) \leq A_+(l) - A_-(l) + A_-(0).
\eeaa
In addition, we have 
\begin{gather*}
  F'(0) \leq A_+(0) = -A_-(0) \leq F'(\infty), \\
 2 F'(l) \leq A_+(l) - A_-(l) \leq 2 F'(\infty),  
\end{gather*}
where the second line follows from~\eqref{eq:Hrel1}.
We deduce that $\| A_+\|_{\infty} \leq 3 \| F'\|_{\infty}$. Similarly, it holds  $\| A_-\|_{\infty} \leq 3\| F'\|_{\infty}$.

\no\rmiii Let us turn now to the proof of~\eqref{eq:Hrel2}. By change of variable, we get
\begin{equation*}
  H(\p_+(l)) - H(\p_+(0)) = \int_{\p_+(0)}^{\p_+(l)} {A_+\left(\psi_+(y)\right) \, dy} = \int_{[0,l]} {A_+(m) \, \p_+'(dm)}.
\end{equation*}
In addition, integration by parts (see, e.g., Bogachev~\cite[Ex.5.8.112]{bogachev-07}) yields that
\begin{equation*}
 \int_{[0,l]} {A_+(m) \, \p_+'(dm)} = A_+(l)\p_+(l) - A_+(0)\p_+(0) - \int_{0}^l {A_+'(m) \p_+(m) \,dm}.
\end{equation*}
Using further $H(0) = H(\p_+(0)) - A_+(0)\p_+(0)$, we obtain
 \bea\label{eq:Hrel3}
  H(\p_+(l)) - A_+(l)\p_+(l) &=&  H(0) - \int_0^{l} {e^{\g(z)} \int_z^{\infty} {e^{-\g(m)} \,F''(dm)} \,dz}.
\eea
 Similarly, $H(\p_-(l)) - A_-(l)\p_-(l)$ coincides with the r.h.s. above, which ends the proof.\hfill\quad\qed
\end{proof}

\begin{proof}[of Proposition~\ref{prop:pathineqsur}]
Let us define $u:\R\times\R_+\to \R$ by 
\beaa
 u(x,l) &:=& - A_+(l) x^+ + A_-(l) x^- + A_+(l)\p_+(l) - H(\p_+(l)) + F(l).
\eeaa
\rmi We start by proving that $u(x,l)\geq F(l)-H(x)$ for all $x\in\R$, $l\geq 0$. Clearly, the restriction of $H$ to $\R_+$ (resp. $\R_-$) is a convex function and we have
\beaa
  \lim_{x\uparrow \p_\pm(l)} H'(x) \leq A_\pm(l) \leq  \lim_{x\downarrow \p_\pm(l)} H'(x).
\eeaa 
Thus, it holds 
\begin{equation*}
  H(x) \geq A_+(l)\left( x-\p_+(l)\right) + H(\p_+(l)),\quad \text{for all } x\geq 0 ,l\geq 0.
\end{equation*}
This yields that $u(x,l)\geq F(l)-H(x)$ for all $x\geq 0$, $l\geq 0$. Similarly, we have
\begin{equation*}
  H(x) \geq A_-(l)\left( x-\p_-(l)\right) + H(\p_-(l)),\quad \text{for all } x\leq 0 ,l\geq 0.
\end{equation*}
Using further~\eqref{eq:Hrel2}, we conclude that $u(x,l)\geq F(l)-H(x)$ for all $x\leq 0$, $l\geq 0$.

\no \rmii Let us show next that
\beaa
 u\left(X_T,L_T\right) &=& \int_0^T {\Delta_t \,dX_t},\quad \P -\mbox{a.s. for all } \P\in\Pc.
\eeaa
Using successively It\^o-Tanaka's formula and the relation~\eqref{eq:Hrel1}, we derive
\begin{multline*}
  - A_+(L_T) X_T^+ + A_-(L_t) X_T^- \\
  \begin{aligned}
   & = \int_0^T {\D_t \,dX_t} - \frac{1}{2}\int_0^{L_T} {\left(A_+(l) - A_-(l)\right) \,dl}\\
   & = \int_0^T {\D_t \,dX_t} - F(L_T) + F(0) - \int_0^{L_T} {e^{\g(l)} \int_l^{\infty} {e^{-\g(m)} \,F''(dm)} \,dl}.
  \end{aligned}
\end{multline*}
We deduce that
\begin{multline*}
 u(X_T,L_T) =  \int_0^T {\D_t \,dX_t} + A_+(L_T)\p_+(L_T) - H(\p_+(L_T)) \\
   + F(0) - \int_0^{L_T} {e^{\g(l)} \int_l^{\infty} {e^{-\g(m)} \,F''(dm)} \,dl}.
\end{multline*}
The desired result follows immediately by using~\eqref{eq:Hrel3}.

\no \rmiii We conclude the proof as follows:
\begin{equation*}
 F(L_T) - H(X_T) \leq u(X_T,L_T) =  \int_0^T {\D_t \,dX_t},
\end{equation*}
where the first inequality comes from Part~(i) above.\hfill\quad\qed
\end{proof}

\begin{remark}
 We refer to Section~\ref{sect:stochcontrol} for a comprehensive presentation of the arguments that led us to consider the function $u$ in the proof of Proposition~\ref{prop:pathineqsur}.
\end{remark}

\subsubsection{{\rm Optimality and Duality}}\label{sec:upper bound}

 In view of Corollary~\ref{prop:upper bound}, the duality is achieved once we find a suitable pair $(\p_+^{\mu},\p_-^{\mu})$ such that the corresponding static strategy $H^\m$ satisfies the relation $\mu(H^\m)=P^{\mu}(F(L_T))$. In this section, we use Vallois' solution to the SEP to construct such a pair and to provide optimizers for both $U^{\mu}(F(L_T))$ and $P^{\mu}(F(L_T))$.
 
\no We start by stating a proposition due to Vallois, which provides a solution to the SEP based on the local time. Recall that $(B_t)_{t\ge 0}$ and $(L^B_t)_{t\ge 0}$ denote a Brownian motion and its local time at zero respectively.

\begin{proposition}\label{prop:vallois}
 Under Assumption~\ref{hyp:mu}, there exists a pair $(\p^\m_+,\p^\m_-)$ satisfying Assumption~\ref{hyp:onemarginal-phi} such that the stopping time
\begin{eqnarray*}
 \t^\m& :=& \inf~ \big\{t>0:~ B_t\notin\big ]\p^\m_-(L^B_t),~ \p^\m_+(L^B_t)\big[\big\}
\end{eqnarray*}
provides a solution to the SEP, i.e., $B^{\t^\m}:= (B_{\t^\m\wedge t})_{t\ge 0}$ is uniformly integrable and $B_{\t^\m} \sim\m$.
\end{proposition}

\begin{proof}
 We refer to Vallois~\cite{vallois-83} or Cox, Hobson and Obl{\'o}j~\cite{cox-hobson-obloj-08} for a proof.\hfill\quad\qed
\end{proof}

\no\begin{remark}\label{symmetric}
If $\mu$ admits a positive density $\mu(x)$ w.r.t. the Lebesgue measure, it holds
\begin{eqnarray*}
\phi^{\mu\,\prime}_{\pm}(dl)  &=&\frac{1-\mu\left([\phi^\mu_-(l),\phi^\mu_+(l)]\right)}{2\phi^\mu_{\pm}(l)\m(\phi^\mu_{\pm}(l))} \mathbf{1}_{[0,\infty[}(l) \,dl.
\end{eqnarray*}
If we assume further that $\mu$ is symmetric, then $\p_\pm^\m=\pm \p^\m$ and
\begin{eqnarray*}
 \psi^\m(x)&=&\int_0^x {\frac{y \mu(y)}{\mu\left([y,\infty[\right) }\,dy}, \quad \text{for all } x \geq 0\label{symmetriceq} 
\end{eqnarray*}
where $\psi^\m$ denotes the inverse of $\p^\m$.
\end{remark}

\no The following theorem, which is the main result of this section, shows that the pair $(\p_+^{\mu},\p_-^{\mu})$ given by Vallois yields the duality and the optimizers for both $U^{\mu}(F(L_T))$ and $P^{\mu}(F(L_T))$.

\begin{theorem}\label{thm:optimality}
Under Assumptions~\ref{hyp:F} and \ref{hyp:mu}, there is no duality gap, i.e.,
\beaa
  P^{\mu}(F(L_T)) = U^{\mu}(F(L_T)) = \mu(H^\m),
\eeaa
 where $H^\m$ is constructed by \eqref{def:H} from $(\p_+^\m,\p_-^\m)$. In addition, there exists an optimizer $\P^\m$ for  $P^{\mu}(F(L_T))$ such that
 \bea \label{eq:optimalheding}
  \int_0^T \D^\m_t dX_t ~+~ H^\m(X_T)&=&F(L_T),\quad \P^\m-\mbox{a.s.},
\eea
 where the process $\D^\m$ is given by \eqref{eq:delta1} with $(\p_+^\m,\p_-^\m)$.
\end{theorem}

\begin{proof}
 \rmi We start by constructing a candidate for the optimizer $\P^\m$. Denote by $\P^{\mu}$ the law of the process $Z=(Z_t)_{0\le t\le T}$ given by
 \beaa
  Z_t &:=& B_{\t^\m\wedge\frac{t}{T-t}} \quad\text{for all } t\in[0,T].
 \eeaa
 The process $Z$ clearly is a continuous martingale w.r.t. its natural filtration such that $Z_T\sim\mu$. In other words, the probability measure $\P^\m$ belongs to $\Pc^{\mu}$.
 
 \no\rmii Let us turn now to the proof of~\eqref{eq:optimalheding}. We define $u^\m: \R\times\R_+\to\R$ by
\beaa
   u^\m(x,l) &:=& - A^{\m}_+(l) x^+ + A^{\m}_-(l) x^- + A^{\m}_+(l)\p^\m_+(l) - H^\m(\p^\m_+(l)) + F(l).
\eeaa
where $A^{\m}_{\pm}$ is given by~\eqref{def:A}--\eqref{def:H'(0)} with $(\p_+^\m,\p_-^\m)$.
Since the local time is invariant under time-change, we have 
\begin{equation*}
 X_T=\p^\mu_+(L_T)\mathbf{1}_{\{X_T>0\}} + \p^\mu_-(L_T)\mathbf{1}_{\{X_T<0\}},\quad \P^\mu-\mbox{a.s.}
\end{equation*}
Notice that $\P^\mu(X_T=0) = \mu(\{0\}) = 0$ in view of Assumption~\ref{hyp:mu}.
Thus, using~\eqref{eq:Hrel2} for the case $X_T<0$, it holds
\beaa
  u^\m(X_{T},L_{T}) &=& F(L_{T}) - H(X_{T}),\quad \P^\m-\text{a.s.}
\eeaa
Further, Part~(ii) of the proof of Proposition~\ref{prop:pathineqsur} ensures that
\beaa
  u^\m(X_{T},L_{T}) &=&  \int_0^T \D^\m_t \,dX_t,\quad \P^\m-\text{a.s.}
\eeaa
Notice that the pair $(\p_+^\mu,\p_-^\m)$ satisfies Assumption~\ref{hyp:onemarginal-phi} in view of Proposition~\ref{prop:vallois}.

\no\rmiii To conclude, it remains to show that $\E^{\P^\m}[F(L_{T})] =  \mu(H^\m)$. This is achieved by taking expectation in~\eqref{eq:optimalheding} and using Lemma~\ref{lem:optimality2} below.\hfill\quad\qed
\end{proof}

\begin{lemma}\label{lem:optimality2}
 With the notations of Theorem~\ref{thm:optimality}, it holds
\beaa
  \E^{\P^\m}\left[\int_0^T \D^\m_t dX_t\right] &=& 0.
 \eeaa
\end{lemma}

\begin{proof}
 This result is a slight extension of Lemma~2.1 in Cox, Hobson and Obl{\'o}j~\cite{cox-hobson-obloj-08}. The proof relies on similar arguments, which we repeat here for the sake of completeness. Using the invariance of the local time under time-change, we observe first that the desired result is equivalent to 
\beaa
  \E\left[\int_0^{\tau^\m} {\left(A^\mu_+(L^B_s) \mathbf{1}_{\{B_s>0\}} + A^\mu_-(L^B_s) \mathbf{1}_{\{B_s\leq 0\}}\right) \,dB_s}\right] &=& 0.
\eeaa
For the sake of clarity, we omit the index $\m$ in the notations and we denote $L$ instead of $L^B$ in the rest of the proof. Let $\sigma_n := \inf\{t\geq 0 :\, |B_t|\geq n\}$, $\rho_m := \inf\{t\geq 0 :\, L_t \geq m\}$, $\t_{n,m} := \t\wedge\sigma_n\wedge\rho_m$ and $\t_{n} := \t\wedge\sigma_n$. We also denote 
\beaa
  M_t &:= &\int_0^{t} {\left( A_+(L_s) \mathbf{1}_{\{B_s>0\}} + A_-(L_s) \mathbf{1}_{\{B_s\leq 0\}}\right) \,dB_s}, \quad \text{ for all }t\geq0.
\eeaa
 From It\^o-Tanaka's formula, it follows that 
\beaa
   M_t &=& A_+(L_t) B_t^+ - A_-(L_t) B_t^- - \frac{1}{2} \int_0^t {\left(A_+(L_s) - A_-(L_s)\right) \,dL_s}.
\eeaa
 We deduce that the stopped local martingale $M^{\t_{n,m}}$ is bounded. Hence, it is a uniformly integrable martingale and we have 
 \begin{multline*}
  \E\left[ \frac{1}{2} \int_0^{\t_{n,m}} {\big(A_+(L_s) - A_-(L_s)\big) \,dL_s} \right] \\
  \begin{aligned}
    & = \E\left[A_+(L_{\t_{n,m}}) B_{\t_{n,m}}^+ - A_-(L_{\t_{n,m}}) B_{\t_{n,m}}^-\right] \\
    & =~~ \E\left[\left(A_+(L_{\t_{n}}) B_{\t_{n}}^+ - A_-(L_{\t_{n}}) B_{\t_{n}}^-\right) \mathbf{1}_{\{\tau_n < \rho_m\}}\right],  
  \end{aligned}
 \end{multline*}
 where the last equality follows from $B_{\rho_m}=0$. 
 It yields that
 \begin{multline*}
  \E\left[ \frac{1}{2} \int_0^{\t_{n,m}} {\big(\left(A_+(L_s) - A_+(0)\right) - \left(A_-(L_s) - A_-(0)\right) \big) \,dL_s} \right] \\
     = \E\left[\left(\left(A_+(L_{\t_{n}}) - A_+(0)\right) B_{\t_{n}}^+ - \left(A_-(L_{\t_{n}}) - A_-(0)\right) B_{\t_{n}}^-\right) \mathbf{1}_{\{\t_n < \rho_m\}}\right].   
 \end{multline*}
 By the monotone convergence theorem, as $m$ tends to infinity, we obtain
 \begin{multline*}
  \E\left[ \frac{1}{2} \int_0^{\t_{n}} {\big(\left(A_+(L_s) - A_+(0)\right) - \left(A_-(L_s) - A_-(0)\right)\big) \,dL_s} \right] \\
    = ~~\E\left[\left(A_+(L_{\t_{n}}) - A_+(0)\right) B_{\t_{n}}^+ - \left(A_-(L_{\t_{n}}) - A_-(0)\right) B_{\t_{n}}^- \right].
 \end{multline*} 
 Then, as $n$ tends to infinity, the l.h.s. converges, again by the monotone convergence theorem, to 
\beaa
  \E\left[ \frac{1}{2} \int_0^{\t} {\big(\left(A_+(L_s) - A_+(0)\right) - \left(A_-(L_s) - A_-(0)\right)\big) \,dL_s} \right].
\eeaa
 As for the r.h.s., using the fact that $A_{\pm}$ are bounded and $(B^{\pm}_{t\wedge\t})_{t\geq 0}$ are uniformly integrable, it converges to 
\beaa
  \E\left[\left(A_+(L_{\t}) - A_+(0)\right) B_{\t}^+ - \left(A_-(L_{\t}) - A_-(0)\right) B_{\t}^- \right]& <&  \infty.
\eeaa
Hence, we obtain
\beaa
  \E\left[ \frac{1}{2} \int_0^{\t} {\big(A_+(L_s) - A_-(L_s)\big) \,dL_s} \right] &=&
     \E\left[A_+(L_{\t}) B_{\t}^+ - A_-(L_{\t}) B_{\t}^- \right],
\eeaa
 where both sides are finite. This ends the proof.\hfill\quad\qed
\end{proof}

\subsection{Robust Subhedging Problem}\label{sec:onemarginal-sub}

In this section, we address the robust subhedging problem. Namely, we derive the lower bound to the no-arbitrage interval and the corresponding optimal subhedging strategy. These results are new to the literature. The idea is to proceed along the lines of Section~\ref{sec:onemarginal-sup}, but to reverse the monotonicity assumption on the functions $\p_+$ and $\p_-$.

\begin{assumption}\label{hyp:onemarginal-phi2}
 $\p_+:]0,\infty[~ \longrightarrow~ ]0,\infty[$ (resp. $\p_-:]0,\infty[~\longrightarrow~ ]-\infty,0[$) is right-continuous and nonincreasing (resp. nondecreasing).
\end{assumption}

\no As in Section~\ref{sec:robusthedging}, we denote by $\psi_\pm$ the right-continuous inverses of $\p_\pm$ and
\begin{eqnarray*}
  \g(l)&:=&\frac{1}{2} \int_0^l\left(\frac{1}{\p_+(m)}-\frac{1}{\p_-(m)}\right) \,dm, \quad \mbox{for all } l>0.
\end{eqnarray*} 
We also define the new functions $A_\pm:\R_+\to\R$  via $(\p_+,\p_-)$ by
\begin{eqnarray}
 A_\pm(l) & := & \pm F'(\infty) - \int_l^{\infty} {\frac{dz}{\p_\pm(z)}} e^{\g(z)}\int_z^{\infty} {e^{-\g(m)} \,F''(dm)}. \label{def:A_sub}
\end{eqnarray}

\subsubsection{{\rm Quasi-Sure Inequality}}

We start by showing the quasi-sure inequality corresponding to the subhedging problem. Together with the second solution provided by Vallois to the SEP, it leads to the solution of the robust subhedging problem. 

\begin{proposition}\label{prop:pathineqsub}
\no Under Assumptions \ref{hyp:F} and \ref{hyp:onemarginal-phi2}, the following inequality holds
\begin{eqnarray}\label{eq:pathineqsub} 
  \int_0^T \D_t \,dX_t + H(X_T) &\le& F(L_T),\quad \P-\mbox{a.s. for all } \P\in\Pc,
\end{eqnarray}
where
\begin{align}
 \D_t & :=   A_+(L_t) \mathbf{1}_{\{X_t > 0\}} - A_-(L_t) \mathbf{1}_{\{X_t\leq 0\}}, & \text{for all } t\in [0,T], \label{eq:delta1_sub}\\
  H(\pm x) & :=  H(0) + \int_0^{\pm x} {A_\pm\left(\psi_\pm (y)\right)\, dy}, & \text{for all } x\geq 0, \label{def:H_sub}\\
  H(0) & :=  F(0) - \int_0^{\infty}{e^{\g(z)} \int_z^{\infty} {e^{-\g(m)} \,F''(dm)} \,dz}. & \label{def:H(0)_sub}
\end{align}
\end{proposition}

\begin{corollary}\label{prop:lowerbound}
 Under Assumptions \ref{hyp:F} and \ref{hyp:onemarginal-phi2}, one has for any centered probability measure $\mu$,
 \begin{eqnarray*}
  I^{\mu}(F(L_T)) \geq D^{\mu}(F(L_T)) \geq \mu(H).
 \end{eqnarray*}
\end{corollary}

\no The proof of Corollary~\ref{prop:lowerbound} is identical to the proof of Corollary~\ref{prop:upper bound}. However, the proof of the quasi-sure inequality~\eqref{eq:pathineqsub} is not completely straightforward and thus we provide some details below.

\begin{proof}[of Proposition~\ref{prop:pathineqsub}] 
Let us show first that once again we have 
\beaa
  \frac{1}{2}\left(A_+(l) - A_-(l)\right) &=& F'(l) + e^{\g(l)} \int_l^{\infty} {e^{-\g(m)} \,F''(dm)}, \label{eq:Hrel1_sub}\\
  H(\p_+(l)) - A_+(l)\p_+(l)  &=& H(\p_-(l)) - A_-(l)\p_-(l). \label{eq:Hrel2_sub}
\eeaa
The first identity follows from Fubini-Tonelli's theorem as was~\eqref{eq:Hrel1} in Lemma~\ref{lem:technic}. As for the second one, by change of variables and integration by parts, we have
\begin{multline*}
 H(\p_+(\infty)) - H(\p_+(l))\\
 \begin{aligned}
    & = \int_{]l,\infty[} {A_+(m) \,\p_+'(dm)} \\  
    & = A_+(\infty) \p_+(\infty) - A_+(l) \p_+(l)  - \int_l^{\infty} e^{\g(z)}\int_z^{\infty} {e^{-\g(m)} \,F''(dm)} \,dy. 
 \end{aligned}
\end{multline*}
Using further $H(0)=H(\p_+(\infty)) - A_+(\infty) \p_+(\infty)$, we obtain
\beaa
  H(\p_+(l)) - A_+(l)\p_+(l) &=& F(0) - \int_0^{l} {e^{\g(z)}\int_z^{\infty} {e^{-\g(m)} F''(dm)}\, dy}.
\eeaa
Similarly, we can show that $H(\p_-(l)) - A_-(l)\p_-(l)$ coincides with the r.h.s. above. 
The rest of the proof follows by repeating the arguments of Proposition~\ref{prop:pathineqsur} using the fact that the restriction of $H$ to $\R_+$ (resp. $\R_-$) is now concave. \hfill\quad\qed
\end{proof}

\subsubsection{{\rm Optimality and Duality}}\label{sec:lowerbound}
Using another solution to the SEP provided by Vallois, we derive the duality and provide optimizers for both $D^{\mu}(F(L_T))$ and $I^{\mu}(F(L_T))$.

\begin{proposition}
 Under Assumption~\ref{hyp:mu}, there exists a pair $(\p^\m_+,\p^\m_-)$ satisfying Assumption~\ref{hyp:onemarginal-phi2} such that $B^{\t^\m}= (B_{\t^\m\wedge t})_{t\ge 0}$ is uniformly integrable and $B_{\t^\m} \sim\m$, where
\begin{eqnarray*}
 \t^\m& :=& \inf~ \big\{t>0:\, B_t\notin ]\p^\m_-(L^B_t), \p^\m_+(L^B_t)[\big\}.
\end{eqnarray*}
\end{proposition}

\begin{proof}
We refer to Vallois~\cite{vallois-92} for a proof.\hfill\quad\qed
\end{proof}

\begin{theorem}\label{thm:optimality_sub}
Under Assumptions~\ref{hyp:F} and \ref{hyp:mu}, there is no duality gap, i.e.,
\beaa
  I^{\mu}(F(L_T)) = D^{\mu}(F(L_T)) = \mu(H^\m),
\eeaa
 where $H^\m$ is constructed by~\eqref{def:H_sub}--\eqref{def:H(0)_sub} from $(\p_+^\m,\p_-^\m)$. In addition, there exists an optimizer $\P^\m$ for  $I^{\mu}(F(L_T))$ such that
 \beaa
  \int_0^T \D^\m_t\, dX_t + H^\m(X_T) &=& F(L_T),\quad\P^\m-\mbox{a.s.},
\eeaa
 where the process $\D^\m$ is given by \eqref{eq:delta1_sub} with $(\p_+^\m,\p_-^\m)$.
\end{theorem}

\no The proof of this result is identical to the proof of Theorem~\ref{thm:optimality}. 
 
\begin{remark}
The assumption that $\mu$ has no mass at zero can be dropped in Theorem~\ref{thm:optimality_sub}. In this case, $\phi_\pm$ can reach zero and we can assume w.l.o.g. that $\psi_+(0)=\psi_-(0)$. Then we need to modify slightly the definitions of $A_\pm$  by replacing the upper bound $\infty$ in the integral term by $\psi_+(0)$. 
\end{remark}

\subsection{On the Stochastic Control Approach}\label{sect:stochcontrol}
 
 As already mentioned, the results of Section~\ref{sec:onemarginal-sup} can be found in Cox, Hobson and Obl{\'o}j~\cite{cox-hobson-obloj-08}. The actual novelty here comes from our approach which is based on the stochastic control theory. In this section, we give some insights on the arguments that led us to consider the quasi-sure inequality~\eqref{eq:pathineqsur}.
 
\no We start by observing that
\begin{align*}
	P^{\mu}(F(L_T)) & = \sup_{\P\in\Pc} \inf_{H\in L^1(\mu)} {\left\{\E^{\P}\left[F(L_T) - H(X_T)\right] + \mu(H)\right\}} \\
	& \leq  \inf_{H\in L^1(\mu)} \sup_{\P\in\Pc} {\left\{\E^{\P}\left[F(L_T) - H(X_T)\right] + \mu(H)\right\}} 
\end{align*}
Further, the Dambis-Dubins-Schwarz theorem implies (formally) that
\begin{equation*}
	P^{\mu}(F(L_T)) \leq  \inf_{H\in L^1(\mu)} \sup_{\tau\in\Tc} {\left\{\E\left[F(L^B_\tau) - H(B_\tau)\right] + \mu(H) \right\}} 
\end{equation*}
 where $\Tc$ is the collection of stopping times $\tau$ such that $B^\tau$ is a uniformly integrable martingale. 
 Inspired by Galichon, Henry-Labord\`ere and Touzi~\cite{galichon-henry-touzi-14}, we study the problem on the r.h.s. above as it turns out to be equivalent to the robust superhedging problem.
 
 \no For any $H \in L^1(\mu)$, we consider the optimal stopping problem
\beaa
  u(x,l) &:=& \sup_{\tau\in\Tc}~ {\E_{x,l}\left[F(L^B_\tau) ~-~ H(B_\tau)\right]},
\eeaa
 where $\E_{x,l}$ denotes the conditional expectation operator 
$\E\left[ \cdot |B_0=x,L^B_0=l\right]$.  
Using the formal representation $d L^B_t = \delta(B_t) \,dt$ where $\delta$ denotes the Dirac delta function, the Hamilton-Jacobi-Bellman (HJB for short) equation corresponding to this optimal stopping problem reads formally as
 \bea\label{eq:pde1}
  \max{\left(F - H - v,\, \frac{1}{2}\partial_{xx}{v} + \d(x)\partial_{l}{v}\right)} &=& 0. 
\eea
 We look for a solution of the form 
\beaa
  v(x,l) &=&
  \begin{cases}
   a(l) x^+ + b(l) x^- + c(l), & \text{if}~(x,l)\in\Dc,\\
   F(l) - H(x), & \text{otherwise},
  \end{cases}
\eeaa
 where $\Dc:=\{(x,l);\, x\in ]\p_-(l),\p_+(l)[\}$ with $\p_\pm$ satisfying Assumption~\ref{hyp:onemarginal-phi}. 
For the sake of simplicity, we assume that $\p_\pm$ are strictly monotone such that $\p_\pm(0)=0$ and $\p_\pm(\infty)=\pm\infty$ and that all the functions involved are smooth enough to allow the calculation sketched below. Differentiating twice w.r.t. to $x$ creates a delta function $\delta(x)$ which cancels out the delta function appearing in the term $\d(x)\partial_{l}{v}$ provided that $a - b = -2c'$. 
 Then we impose the continuity and the smooth fit conditions at the boundary $\partial\Dc$:
 \begin{eqnarray*}
   v(\p_\pm(l),l)=F(l)-H(\p_\pm(l)) &~\mbox{and}~& \partial_x v(\p_\pm(l),l)=-H'(\p_\pm(l)).
 \end{eqnarray*}
 We deduce that
\begin{multline*}
 v(x,l) = - H'(\p_+(l)) x^+ + H'(\p_-(l))x^- \\
    + H'(\p_+(l))\p_+(l) - H(\p_+(l)) +  F(l),\quad \text{for all }(x,l)\in\Dc.
\end{multline*}
 In addition, the function $H$ has to satisfy the following system of ODEs:
 \begin{eqnarray*}
   H'(\p_+) \p_+ - H(\p_+) &=& H'(\p_-) \p_- - H(\p_-), \\ 
   \frac{1}{2} \big(H'(\p_+) - H'(\p_-)\big) &=& F' + H''(\p_+)\p'_+\p_+.
 \end{eqnarray*}
 This system of ODE can be solved explicitly and it characterizes $H'\circ\p_\pm$ as in~\eqref{def:A} up to a constant such that
 \beaa
  \frac{1}{2} \left(H'(0+) - H'(0-)\right) &=& F'(0) + \int_0^{\infty} {e^{-\g(m)} F''(m) \, dm}.
 \eeaa
 Thus $H$ is given by~\eqref{def:H} if we pick $H(0)=F(0)$. 
Conversely, provided that $F$ is convex, $H$ is also convex on $\R_+$ (resp. $\R_-$). It follows that $v$ satisfies the variational inequality~\eqref{eq:pde1} in view of Part~(i) of the proof of Proposition~\ref{prop:upper bound}. Further, we observe that another solution of \eqref{eq:pde1} is given by
\begin{multline*}
 {u}(x,l) := 
   - H'(\p_+(l)) x^+ + H'(\p_-(l))x^- \\
   + H'(\p_+(l))\p_+(l) - H(\p_+(l)) +  F(l),\quad \text{for all }(x,l)\in\R\times\R_+.
\end{multline*}

\no Then, given any solution $w$ to the variational PDE (\ref{eq:pde1}), it is straightforward to derive heuristically a quasi-sure inequality. Indeed, using the formal representation $d L_t = \delta(X_t) \,d\langle X\rangle_t$, It\^o's formula yields for all $\P\in\Pc$,
\beaa
  F(L_T) - H(X_T)  \leq w\left(X_T,L_T\right) \leq  \int_0^T {\partial_x{w}\left(X_s,L_s\right) \,dX_s}, \quad \P-\mathrm{a.s.}
\eeaa
 In particular, if $w$ coincides with $u$, we recover the quasi-sure inequality~\eqref{eq:pathineq}. Further, if we denote by $\P^*$ the distribution of $(B_{\tau\wedge\frac{t}{T-t}})_{t\in[0,T]}$ where 
\beaa
  \tau &:=& \inf\left\{ t\geq 0:\, B_t\notin \big]\p_-(L^B_t), \p_+(L^B_t)\big[\right\},
\eeaa
  we obtain
 \begin{equation*}
  F(L_T) - H(X_T) = u(X_T, L_T) = \int_0^T {\partial_x{u}\left(X_s,L_s\right) \,dX_s},\quad \P^*-\mathrm{a.s.}
 \end{equation*}

\section{Two-Marginal Skorokhod Embedding Problem}\label{sec:two-marginal}\label{sec:embedding2}

\no In this section, we provide a new solution to the two-marginal SEP as an extension of the Vallois embedding. Let $\bu=(\mu_1,\mu_2)$ be a centered peacock, i.e., $\mu_1$ and $\mu_2$ are centered probability distributions such that $\mu_1(f)\leq\mu_2(f)$ for all $f:\R\to\R$ convex. We aim to construct a pair of stopping rules based on the local time such that $\t^{\bu}_1\le \t^{\bu}_2$, $B_{\tau^{\bu}_1}\sim \mu_1$, $B_{\tau^{\bu}_2}\sim \mu_2$ and $B^{\t^{\bu}_2}$ is uniformly integrable. A natural idea is to take the Vallois embeddings corresponding to $\mu_1$ and $\mu_2$. 
However these stopping times are not ordered in general and thus we need to be more careful. For technical reasons, we make the following assumption on the marginals.
 
\begin{assumption}\label{hyp:marginal2}
 $\bu=(\mu_1,\mu_2)$ is a centered peacock such that $\mu_1$ and $\mu_2$ are symmetric and equivalent to the Lebesgue measure.
\end{assumption}

\subsection{Construction}\label{sec:embedding2-main}

For the first stopping time, we take the solution given by Vallois~\cite{vallois-83} that embeds $\mu_1$, i.e., 
\begin{equation*}
 \tau^{\bu}_1 := \inf \Big{\{t > 0: \left|B_t\right|\geq \phi^{\bu}_1(L_t)\Big\}},
\end{equation*}
where $\phi^{\bu}_1:\R_+\to\R_+$ is the inverse of
\begin{eqnarray}
 \psi^{\bu}_1(x):=\int_0 ^x {\frac{y \mu_1(y)}{\mu_1([y,\infty[)}\, dy}, \quad\text{for all }x\geq0.\label{eq:symmetriceq} 
\end{eqnarray}
For the second stopping time,  
we look for an increasing function $\phi^{\bu}_2:\R_+\to\R_+$ such that
\begin{eqnarray*}
 \tau^{\bu}_2:=\inf\Big\{ t \geq\tau^{\bu}_1 \; : \; |B_t | \geq \phi^{\bu}_2(L_t) \Big\}.
\end{eqnarray*}
Notice that $\t_1^{\bu}\leq\t_2^{\bu}$ by definition. In particular, if $\phi^{\bu}_2< \phi^{\bu}_1$ on a non-empty interval, it can happen that $|B_t | \geq \phi^{\bu}_2(L_t)$ for some $t<\tau^{\bu}_1$.
As before, $\p^{\bu}_2$ is defined through its inverse $\psi^{\bu}_2$. Let us denote
\begin{align*}
 x_1 &:= \inf{\left\{x>0:\, \int_0^x {\frac{y \mu_2(y)}{\mu_2([y,\infty[)}\,dy} > \psi^{\bu}_{1}(x)\right\}}.
\end{align*}
Then we set
\begin{equation}\label{eq:psi20}
 \psi^{\bu}_{2}(x) = \int_0^x {\frac{y \mu_2(y)}{\mu_2([y,\infty[)}\,dy}, \quad \text{for all } x \in [0,x_1].
\end{equation}
To ensure that $x_1>0$, we need to assume that $\delta\mu:=\mu_2-\mu_1\leq 0$ on a neighborhood of zero. 
If $x_1=\infty$, the construction is over. This corresponds to the case when the Vallois embeddings are well-ordered. Otherwise, we proceed by induction as follows:

\no\rmi if $x_{2i-1}<\infty$, we denote 
\begin{align*}
 x_{2i} &:= \inf{\left\{x>x_{2i-1}:\, \psi^{\bu}_2(x_{2i-1}) + \int_{x_{2i-1}}^x {\frac{y \delta\mu(y)}{\delta\mu([y,\infty[)} \,dy} < \psi^{\bu}_{1}(x)\right\}}.
\end{align*}
Then we set for all $x\in ]x_{2i-1}, x_{2i}]$,
  \begin{equation}
     \psi^{\bu}_{2}(x) = \psi^{\bu}_2(x_{2i-1}) + \int_{x_{2i-1}}^x {\frac{y \delta\mu(y)}{\delta\mu([y,\infty[)} \,dy}; \label{eq:psi2a}
  \end{equation}
  
\no\rmii if $x_{2i}<\infty$, we denote 
\begin{align*}
 x_{2i+1} &:= \inf{\left\{x>x_{2i}:\, \psi^{\bu}_2(x_{2i}) + \int_{x_{2i}}^x {\frac{y \mu_2(y)}{\mu_2([y,\infty[)} \,dy} > \psi^{\bu}_{1}(x)\right\}}.
\end{align*}
Then we set for all $x\in ]x_{2i}, x_{2i+1}]$,
  \begin{equation}
     \psi^{\bu}_{2}(x) = \psi^{\bu}_2(x_{2i}) + \int_{x_{2i}}^x {\frac{y \mu_2(y)}{\mu_2([y,\infty[)} \,dy}. \label{eq:psi2b}
  \end{equation}
  
\no To ensure that $\psi_2^{\bu}$ is well-defined and increasing, we need to make the following assumption. In particular, the point (iii) below ensures that $x_i<x_{i+1}$ and $\lim_{i\to\infty}{x_i}=\infty$. 

\begin{assumption}\label{hyp:embedding}
 \rmi $\delta\mu\leq 0$ and $\delta\mu\not\equiv 0$ on a neighborhood of zero;\\
 \rmii $\delta \mu > 0$ whenever $\psi^{\bu}_1<\psi^{\bu}_2$;\\
 \rmiii $x_i= \infty$ for some $i\geq1$.
\end{assumption}

\begin{theorem}\label{thm:vallois2}
 Under Assumptions \ref{hyp:marginal2} and \ref{hyp:embedding}, if $\psi^{\bu}_2$ is given by~\eqref{eq:psi20}--\eqref{eq:psi2b}, then $B^{\t^{\bu}_2}$ is uniformly integrable and $B_{\t^{\bu}_2} \sim \mu_2$.
\end{theorem}

\no The proof of this result will be performed in the next section. Notice already that  
both points (i) and (ii) of Assumption~\ref{hyp:embedding} are to a certain extent necessary conditions to ensure the existence of an increasing function $\psi^{\bu}_2$ that solves the two-marginal SEP as above. 
See Remark~\ref{rem:embedding} below for more details. 
This suggests that one needs to relax the monotonicity assumption on $\psi_2^{\bu}$ in order to iterate the Vallois embedding for a larger class of marginals. 
However, our approach does not allow to compute the function $\psi_2^{\bu}$ in this general setting.

\subsection{Proof of Theorem~\ref{thm:vallois2}}
 
\no Let us start by a technical lemma which is a key step in the proofs of Theorem~\ref{thm:vallois2} and Theorem~\ref{th:markov} below.
  
\begin{lemma} \label{lemf}
 Let $\p_\pm$ be given as in Assumption~\ref{hyp:onemarginal-phi} and denote 
 \begin{equation*}
  \tau:=\inf\left\{ t >0: B_t\notin \left]\phi_-(L_t),\phi_+(L_t)\right[\right \}.
 \end{equation*}
 For any $f:\R\times\R_+\to\R$ bounded such that $l\mapsto f(\p_\pm(l),l)$ is of bounded variation, it holds for all $l\geq0$, $x\in ]\p_-(l),\p_+(l)[$,
\beaa
  \Eb_{x,l}[f(B_{\tau},L_{\tau})] &=& 
   \frac{f\left(\phi_+(l),l\right)-c(l)}{\phi_+(l)}x^+ -\frac{f\left(\phi_-(l),l\right)-c(l)}{\phi_-(l)}x^-+ c(l),
\eeaa
 where $\E_{x,l}$ denotes the conditional expectation operator 
 $\E\left[ \cdot |B_0=x,L_0=l\right]$ and
 \beaa
 c(l)&=& \frac{e^{\gamma(l)}}{2} \int_l^{\infty} \left(\frac{f(\p_+(m),m)}{\p_+(m)}-\frac{f(\p_-(m),m)}{\p_-(m)}\right)e^{-\gamma(m)} \,dm.
 \eeaa
 \end{lemma}

\begin{proof}
 Let $(M_t)_{t\geq 0}$ be the process given by
 \begin{equation*}
  M_t  =  \frac{f\left(\phi_+(L_t),L_t\right)-c(L_t)}{\phi_+(L_t)}B_t^+ -\frac{f\left(\phi_-(L_t),L_t\right)-c(L_t)}{\phi_-(L_t)}B_t^-+ c(L_t).
 \end{equation*}
 By applying It\^o-Tanaka's formula and using further
 \begin{equation*}
  c'(l) + \frac{1}{2}\left(\frac{f\left(\phi_+(l),l\right)-c(l)}{\phi_+(l)} - \frac{f\left(\phi_-(l),l\right)-c(l)}{\phi_-(l)}\right) = 0,
 \end{equation*}
 we deduce that
 \begin{multline*}
  M_t  =  M_0 + \int_0^t {\frac{f\left(\phi_+(L_s),L_s\right)-c(L_s)}{\phi_+(L_s)}\mathbf{1}_{\{B_s>0\}} \,dB_s}\\
  + \int_0^t {\frac{f\left(\phi_-(L_s),L_s\right)-c(L_s)}{\phi_-(L_s)}\mathbf{1}_{\{B_s\le 0\}} \,dB_s}.
 \end{multline*} 
 Hence, the process $M$ is a local martingale. Further, the stopped process $M^{\tau}$ is bounded since $\|c\|_{\infty}\le\|f\|_{\infty}$ and $|B^{\pm}_{\t\wedge t}|\le |\phi_{\pm}(L_{\t\wedge t})|$. It follows that 
 \begin{equation*}
  \E_{x,l}\left[M_\tau\right] = M_0 
   = \frac{f\left(\phi_+(l),l\right)-c(l)}{\phi_+(l)}x^+ -\frac{f\left(\phi_-(l),l\right)-c(l)}{\phi_-(l)}x^-+ c(l). 
 \end{equation*}
 To conclude, it remains to see that $M_\tau=f(B_\tau,L_\tau)$ by definition.\hfill\quad\qed
\end{proof}
 
 \no We are now in a position to complete the proof of Theorem~\ref{thm:vallois2}. For the sake of clarity, we omit the index $\bu$ in the notations and we split the proof in three steps.
 
 \no\textsl{First step.} We start by showing that the distribution of $B_{\tau_2}$ admits a density w.r.t. the Lebesgue measure. Let us assume here that the function $\psi_2$ is increasing and satisfies $\g_2(0+)=0$ and $\g_2(\infty)=\infty$ where $\g_2:=\int_0^{\cdot}{\frac{1}{\p_2(m)}\,dm}$. This will be proved in Step~3 below. Notice first that the distribution of $B_{\tau_2}$ is symmetric by construction. By the strong Markov property and Lemma~\ref{lemf}, if we take the function $\l(y)=\mathbf{1}_{\{y\leq x\}}$ for some $x\geq0$, it holds
 \begin{equation*}
  \E\left[\l\left(\left|B_{\tau_2}\right|\right)\right] = \E\left[f\left(B_{\tau_1},L_{\tau_1}\right)\right] = \int_0^{\infty} {\frac{f\left(\p_1(l),l\right)}{\p_1(l)} e^{-\g_1(l)} \,dl},
 \end{equation*}
 where
 \begin{gather*}
  f(y,l)  := 
  \begin{cases}
    \frac{\l(\p_2(l)) - c(l)}{\p_2(l)} |y| + c(l), & \text{if }  |y|<\p_2(l), \\
    \lambda(|y|), &  \text{otherwise},
  \end{cases}
  \\  
  c(l)  := e^{\gamma_2(l)} \int_l^{\infty} {\frac{\lambda(\p_2(m))}{\p_2(m)} e^{-\gamma_2(m)} dm}.
 \end{gather*} 
 By a straightforward calculation, we get
 \begin{gather*}
  c(l) = \left(1 - e^{\g_2(l) - \g_2(\psi_2(x))}\right) \mathbf{1}_{\{l\leq \psi_2(x)\}}, \\
   f\left(\p_1(l),l\right) =
   \begin{cases}
     \left(1 + e^{\g_2(l) - \g_2(\psi_2(x))} \left(\frac{\p_1(l)}{\p_2(l)} - 1\right)\right) \mathbf{1}_{\{l \leq \psi_2(x)\}},& \text{if } \p_1(l) < \p_2(l), \\
      \mathbf{1}_{\{l\leq \psi_1(x)\}}, & \text{otherwise.}
   \end{cases}
 \end{gather*}
 Hence, we obtain
 \begin{multline*}
  \P\left(\left|B_{\tau_2}\right|\leq x\right) 
  =  e^{-\g_2(\psi_2(x))} \int_{\{\p_1<\p_2\}} {\mathbf{1}_{\{l\leq\psi_2(x)\}} \,de^{\g_2(l)-\g_1(l)}} \\
  - \int_{\{\p_1<\p_2\}} {\mathbf{1}_{\{l\leq\psi_2(x)\}} \,de^{-\g_1(l)}} -  \int_{\{\p_1\geq\p_2\}} {\mathbf{1}_{\{l\leq\psi_1(x)\}} \,de^{-\g_1(l)}}.
 \end{multline*}
 We deduce by direct differentiation of the identity above that the distribution of $B_{\tau_2}$ admits a density $\nu$ w.r.t. the Lebesgue measure given by
 \begin{equation}\label{eq:density}
  \nu(x) = 
    \begin{cases}  
      \displaystyle\frac{\psi_2'(x)}{2x} \left(S(x) + e^{-\g_1(\psi_2(x))}\right)  & \text{if } \psi_1(x)>\psi_2(x), \\
      \displaystyle\frac{\psi_1'(x)}{2x}e^{-\g_1(\psi_1(x))} + \frac{\psi_2'(x)}{2x} S(x) & \text{otherwise},
     \end{cases}
 \end{equation}
 where
 \begin{equation*}
  S(x) := - e^{-\g_2(\psi_2(x))} \int_{\{\p_1<\p_2\}} {\mathbf{1}_{\{l\leq\psi_2(x)\}}\,d{e^{\g_2(l)-\g_1(l)}}}. 
 \end{equation*}

 \no\textsl{Second step.} Let us show now that $\nu$ coincides with $\mu_2$ when $\p_2$ is defined as~\eqref{eq:psi20}--\eqref{eq:psi2b}. Notice first that the relation~\eqref{eq:symmetriceq} yields that
 \begin{equation*}
  e^{-\g_1(\psi_1(x))} = e^{-\int_0^x{\frac{\psi_1'(y)}{y}\,dy}} = 2 \mu_1([x,\infty[),\quad \text{for all }x\in\R_+.
 \end{equation*}
 Using further the identities~\eqref{eq:psi2a} and~\eqref{eq:psi2b}, it follows that
 \begin{equation*} 
  \nu(x) = 
    \begin{cases}  
      \displaystyle\frac{\mu_2(x)}{\mu_2([x,\infty[)} S_{2i}(x) & \text{if } x\in ]x_{2i},x_{2i+1}], \\
      \displaystyle\mu_1(x) + \frac{\delta\mu(x)}{\delta\mu([x,\infty[)} S_{2i+1}(x) & \text{if } x\in ]x_{2i+1},x_{2i+2}].
     \end{cases}
 \end{equation*}
 where we denote $l_i:=\psi_1(x_i)=\psi_2(x_i)$ and
 \begin{equation*}
  S_i(x) := \frac{e^{-\gamma_2(\psi_2(x))}}{2} \sum_{j=0}^i {(-1)^j e^{\gamma_2(l_j) - \gamma_1(l_j)}}.
 \end{equation*} 
 To conclude, it remains to show that $S_{2i}(x) = \mu_2([x,\infty[)$ for all $x\in[x_{2i},x_{2i+1}]$ and $S_{2i+1}(x) = \delta\mu([x,\infty[)$ for all $x\in[x_{2i+1},x_{2i+2}]$. 
 As a by-product, this proves that
 \begin{equation}\label{eq:deltamu}
  \delta\mu([x,\infty[) > 0,\quad \text{for all } x\in[x_{2i+1},x_{2i+2}],
 \end{equation}
 and thus $\psi_2$ is well-defined. 
 For $i=0$, it follows from the relation~\eqref{eq:psi20} that
 \begin{equation*}
  S_0(x) = \frac{e^{-\gamma_2(\psi_2(x))}}{2} = \mu_2([x,\infty[), \quad \text{for all } x\in[0,x_1].
 \end{equation*}
 Assume that $S_{2i}(x) = \mu_2([x,\infty[)$ for all $x\in[x_{2i},x_{2i+1}]$. It results from the relation~\eqref{eq:psi2a} that 
 \begin{equation*}
  e^{-\gamma_2(\psi_2(x)) + \gamma_2(l_{2i+1})} = \frac{\delta\mu([x,\infty[)}{\delta\mu([x_{2i+1},\infty[)}, \quad \text{for all }x\in [x_{2i+1},x_{2i+2}].
 \end{equation*}
 Hence, we deduce that for all $x\in [x_{2i+1},x_{2i+2}]$,
 \begin{equation*}
  S_{2i+1}(x) 
     = e^{-\gamma_2(\psi_2(x)) + \gamma_2(l_{2i+1})} \left(S_{2i}\left(x_{2i+1}\right) - \frac{e^{-\gamma_1(l_{2i+1})}}{2}\right) = \delta\mu([x,\infty[).
 \end{equation*}
 Further, it results from the relation~\eqref{eq:psi2b} that 
 \begin{equation*}
  e^{-\gamma_2(\psi_2(x)) + \gamma_2(l_{2i+2})} = \frac{\mu_2([x,\infty[)}{\mu_2([x_{2i+2},\infty[)}, \quad \text{for all }x\in [x_{2i+2},x_{2i+3}].
 \end{equation*}
 Hence, we deduce that for all $x\in [x_{2i+2},x_{2i+3}]$,
 \begin{equation*}
  S_{2i+2}(x) 
     = e^{-\gamma_2(\psi_2(x)) + \gamma_2(l_{2i+2})} \left(S_{2i+1}\left(x_{2i+2}\right) + \frac{e^{-\gamma_1(l_{2i+2})}}{2}\right) = \mu_2([x,\infty[).
 \end{equation*}

\no \textsl{Third step.} We are now in a position to complete the proof. The relation~\eqref{eq:psi2b} clearly imposes that $\psi_2$ is increasing on every interval such that $\psi_1>\psi_2$. Under Assumption~\ref{hyp:embedding} (ii), the relation~\eqref{eq:psi2a} together with~\eqref{eq:deltamu} ensures that $\psi_2$ is increasing on every interval such that $\psi_1<\psi_2$.
In addition, it follows immediately from~\eqref{eq:psi20} that $\g_2(0+)=0$. Further, in view of Assumption~\ref{hyp:marginal2} (iii), one easily checks by a straightforward calculation that $\g_2(\infty)=\infty$. 
It remains to prove that  the stopped process $B^{\t_2}$ is uniformly integrable. Since $|B_{t\wedge\t_2}|\leq |B_{\tau_2}|$ for all $t\geq 0$, the uniform integrability follows immediately from the assumption that $\mu_2$ admits a finite first moment. \hfill\quad\qed

\begin{remark}\label{rem:embedding}
 The first step of the proof does not rely on the specific form of the increasing function $\psi_2$. For this reason, the relation~\eqref{eq:density} sheds new light on Assumption~\ref{hyp:embedding}.
 For instance, if $\psi_1\leq\psi_2$ near zero, then we see that $\nu=\mu_1$ near zero. Else $\psi_2=\int_{0}^{\cdot}{\frac{y\nu(y)}{\nu([y,\infty[)}\,dy}$ near zero. Thus, one cannot expect to solve the two-marginal SEP as above, if $\delta\mu\geq 0$ and $\delta\mu\not\equiv 0$ on a neighborhood of zero. In addition, since $S\geq 0$ by construction, we deduce that $\psi_2$ is increasing if and only if $\nu-\mu_1>0$ whenever $\psi_1<\psi_2$.
\end{remark}

\subsection{A Numerical Example}\label{sec:numerical}

\no As an example, we consider the pair of symmetric densities $(\mu_1,\mu_2)$ given by
\begin{gather*}
 \mu_1(x) := e^{-2 x}, \quad \text{for all } x\geq 0,\\
 \mu_2(x) :=
  \begin{cases}
   \frac{5}{2} x^{3} e^{-\frac{5x^{4}}{4}}, & \text{if } 0\leq x\leq 1,  \\
   e^{-2 x } + \delta \mu(1)x^{\alpha-2} e^{-\frac{\alpha (x^{\alpha -1}-1)}{\alpha-1}}, & \text{if } x > 1,
  \end{cases}
\end{gather*}
where $\alpha$ is a parameter satisfying $\delta\mu(1)= \alpha \delta \mu([1,\infty[)$. 
%One can easily check by direct calculation that the marginals $\mu_1$ and $\mu_2$ are increasing in the convex order. 
The corresponding embedding maps $\psi^{\bu}_1$ and $\psi^{\bu}_2$ are given by
\beaa
 \psi^{\bu}_1(x):= x^2 &~~\mbox{and}~~& \psi^{\bu}_2(x):= 
  \begin{cases}
    x^5, & \text{if } 0\leq x\leq 1,  \\
    x^\alpha, & \text{if } x > 1.
  \end{cases}
\eeaa
All the assumptions in Theorem \ref{thm:vallois2} are satisfied as can be seen in Figure~\ref{figSEP}.

\begin{figure}[htp]
 \centering
  \includegraphics[width=0.75\linewidth]{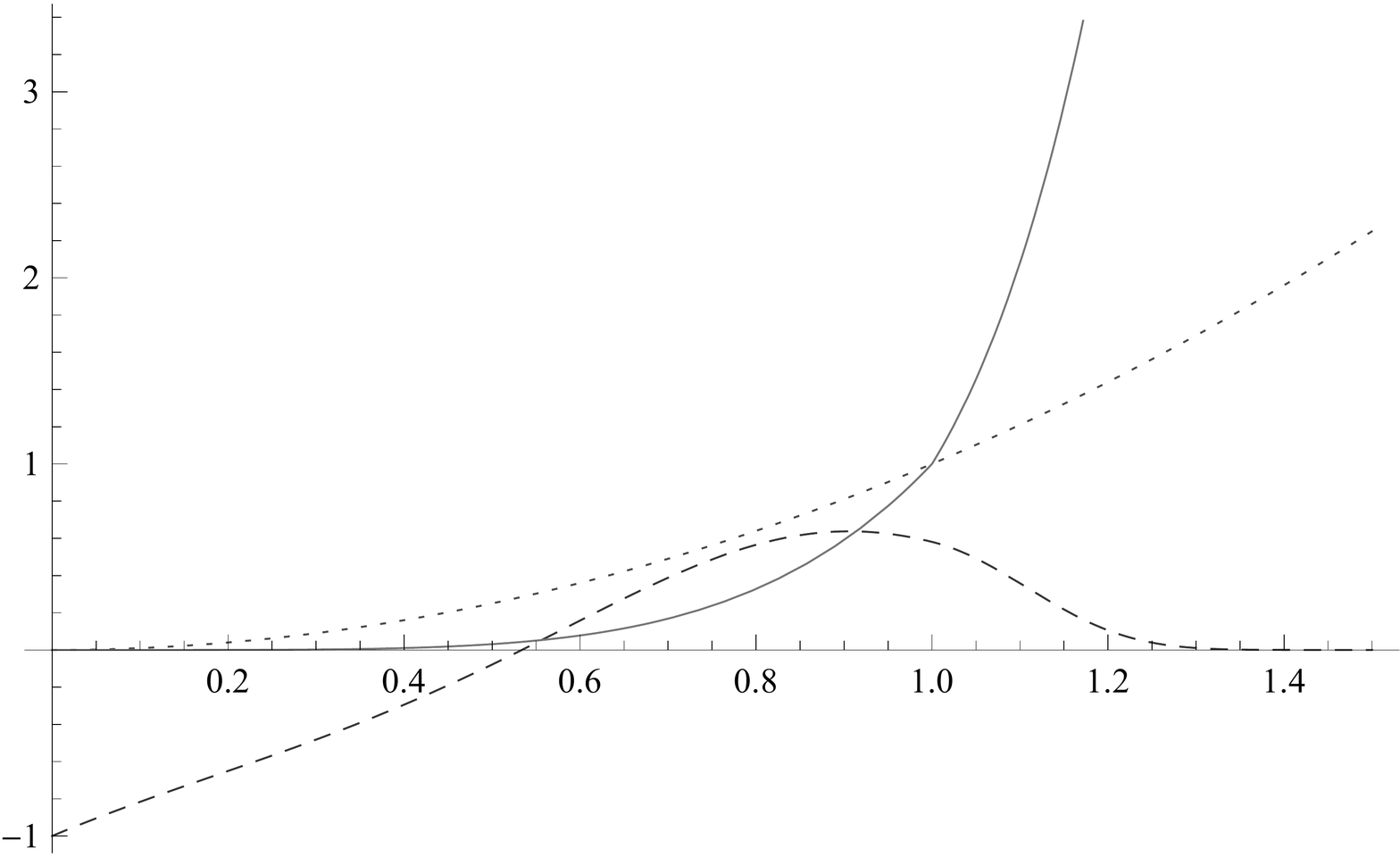} 
  \caption{Embedding functions $\psi_1$ (dotted line), $\psi_2$ (solid line) and the difference of densities $\delta \mu$ (dashed line).} \label{figSEP}
\end{figure}

\no In Figure \ref{fig2}, we provide a comparison of the analytical cumulative distributions of $B_{\tau^{\bu}_1}$ and $B_{\tau^{\bu}_2}$ with their Monte-Carlo estimations using $2^{17}$ paths. We find a very good match except on a neighborhood of zero. Note that the simulation of $\tau^{\bu}_1$ and $\tau^{\bu}_2$ are quite difficult as we need to simulate the local time of a Brownian motion, which is a highly irregular object. We have chosen to simulate the local time $L_{k \Delta t}$ at a time step $k \Delta t$ using
\beaa 
 L_{k \Delta t}-L_{(k-1) \Delta t}&=&\frac{\Delta t}{2\epsilon} \mathbf{1}_{\{B_{(k-1) \Delta t} \in [-\epsilon,\epsilon]\}} 
\eeaa 
with $\epsilon=0.04$ and $\Delta t=1/4000$. Since the derivatives of $\p^{\bu}_1$ and $\p^{\bu}_2$ are infinite at zero, the accuracy of our Monte Carlo estimations near zero depends strongly on the discretization of the local time, which explains the small mismatch in Figure \ref{fig2}. 

\begin{figure}[htp]
 \centering
  \includegraphics[width=0.75\linewidth]{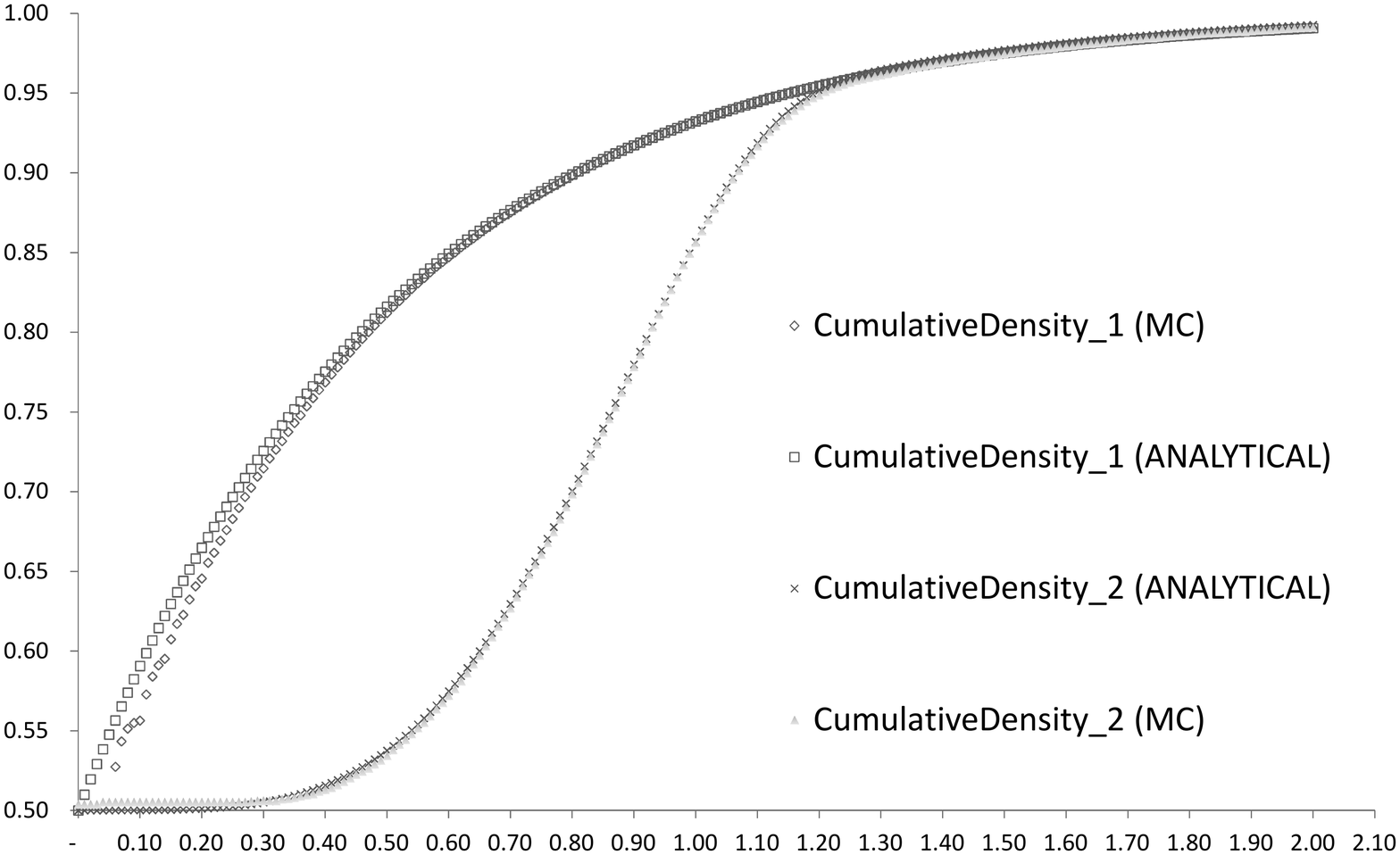}
 \caption{Analytic cumulative distributions of $\mu_1$ ($\square$) and $\mu_2$ ($\times$), and their Monte Carlo approximations for $\mu_1$ ($\diamond$) and $\mu_2$ ($\triangle$).}
 \label{fig2}
\end{figure}

\section{A Remarkable Markov Martingale with Full Marginals}\label{sec:multi}

\no In this section, we exhibit a remarkable Markov martingale under the assumption that all the marginals of the process are known. In particular, it provides a new example of fake Brownian motion.
 
 \subsection{Infinitesimal Generator}
 
We consider that all the marginals $(\mu_t)_{0\leq t\leq T}$ of the process are known. For the sake of simplicity, we assume that $\mu_t$ is symmetric and equivalent to the Lebesgue measure for every $t\in[0,T]$. Denote by $\t_t$ (resp. $\p_t$) the stopping time (resp. the map) given by Vallois~\cite{vallois-83} that embed the distribution $\mu_t$. 

\begin{assumption}\label{cvallois} 
 For all $0\le s\le t\le T$, $\t_s\le \t_t$, or equivalently, $\p_s\le~\p_t$.
\end{assumption}

\no The next result gives the generator of the Markov process $(B_{\t_t})_{0\le t\le T}$. It is analogous to the study of Madan and Yor~\cite{madan-yor-02} with the Az\'ema-Yor solution to the SEP. In particular, the process $(B_{\t_t})_{0\le t\le T}$ is a pure jump process, which corresponds to an example of local L\'evy model introduced in Carr et al. \cite{carr-al-04}.

\begin{theorem}\label{th:markov} 
  Under Assumption \ref{cvallois}, $(B_{\t_t})_{0\le t\le T}$ is an inhomogeneous Markov martingale whose generator is given by
  \begin{multline*}
    \Lc_t f(x) =  - \frac{\partial_t \ps_t(|x|)}{\partial_x\ps_t(|x|)} \bigg( \mathrm{sgn}(x)f'(x) \\
     - \frac{e^{\g_t(\ps_t(|x|))}}{2|x|} \int_{|x|}^{\infty} {\Big( { f(y) +f(-y)} - 2f(x) \Big) d e^{-\g_t(\ps_t(y))}}\bigg),
  \end{multline*} 
  where $\gamma_t(l):=\int_0^l{\frac{1}{\phi_t(m)}\, dm}$, $l\geq0$.
   \end{theorem}

 \begin{proof}
 The process $(B_{\t_t})_{0\le t\le T}$ is clearly an inhomogeneous Markov martingale, see Madan and Yor~\cite{madan-yor-02} for more details.
It remains to compute the generator. For each $0\le t<s\le T$, we denote 
\begin{eqnarray*}
\Eb\big[f(B_{\t_s})\,\big|\,B_{\t_t}=x\big] = \Eb\big[f(B_{\t_s})\,\big|\,B_{\t_t}=x, L_{\t_t}=\ps_t(|x|)\big]&=:&v\big(x,\ps_t(|x|)\big).
\end{eqnarray*}
Then it follows from Lemma~\ref{lemf} that
\beaa
v\big(x,\ps_t(|x|)\big)&=&a_s\circ\ps_t(|x|)\,x^+ + b_s\circ \ps_t(|x|)\,x^- + c_s\circ\ps_t(|x|)
\eeaa
with
\beaa
a_s(l)&:=&\frac{f\big(\p_s(l)\big)-c_s(l)}{\p_s(l)}, \\
b_s(l)&:=&\frac{f\big(-\p_s(l)\big) - c_s(l)}{\p_s(l)}, \\
c_s(l)&:=&\frac{e^{\g_s(l)}}{2}\int_l^{\infty}\frac{f\big(\p_s(m)\big)+f\big(-\p_s(m)\big)}{\p_s(m)} e^{-\g_s(m)} \,dm.
\eeaa
By definition, the generator is given by, for all $x\in\R$,
\begin{eqnarray*}
\Lc_t f(x) &=& \partial_t a_t\circ \ps_t(|x|)\,x^+ + \partial_t b_t\circ\ps_t(|x|)\, x^- + \partial_tc_t\circ\ps_t(|x|).
\end{eqnarray*}
Differentiating the relation $\p_t\circ\psi_t(|x|)=|x|$ w.r.t. $t$, we obtain
\begin{equation*}
 \partial_t\p_t\circ\psi_t(|x|) + \frac{\partial_t\ps_t(|x|)}{\partial_x {\ps_t}(|x|)} = 0.
\end{equation*}
Using the formula above, a straightforward calculation yields that for all $x\ge 0$,
\begin{eqnarray*}
x\, \partial_ta_t\circ\ps_t(x) & = &  \frac{\partial_t\ps_t(x)}{\partial_x {\ps_t}(x)} \left(\frac{f(x) - c_t\circ\ps_t(x)}{x} - f'(x)\right) -\partial_tc_t\circ\ps_t(x).
\end{eqnarray*}
Similarly, it holds for all $x<0$,
\begin{multline*}
 -x\,\partial_tb_t\circ\ps_t(-x) \\
 = \frac{\partial_t\ps_t(-x)}{\partial_x {\ps_t}(-x)} \left(\frac{f(x) - c_t\circ\ps_t(-x)}{-x} + f'(x)\right) - \partial_tc_t\circ\ps_t(-x).
\end{multline*}
Hence, we obtain
\begin{eqnarray*}
\Lc_t f(x) &=& \frac{\partial_t \ps_t(|x|)}{\partial_x{\ps_t}(|x|)} \left( \frac{f(x)-c_t\circ\ps_t(|x|)}{|x|} - \mathrm{sgn}(x) f'(x) \right). 
\end{eqnarray*}
The desired result follows by using further
\beaa
c_t\circ\ps_t(|x|) &=& \frac{e^{\g_t\circ\ps_t(|x|)}}{2}\int_{|x|}^{\infty}{\frac{f(y)+f(-y)}{y} \partial_x {\ps_t}(y) e^{-\g_t\circ\ps_t(y)}}\,dy \\
&=&-\frac{e^{\g_t\circ\ps_t(|x|)}}{2}\int_{|x|}^{\infty}\big(f(y)+f(-y)\big)de^{-\g_t\circ\ps_t(y)}. \hfill \quad    \qed
\eeaa
\end{proof}

 \subsection{Fake Brownian Motion}

\no As an application, we provide a new example of fake Brownian motion. If $(\m_t)_{0\le t\le T}$ is a continuous Gaussian peacock, i.e.,
\begin{equation*}
 \mu_t(x)=\frac{1}{\sqrt{2\pi t}}e^{-\frac{x^2}{2t}}, \quad \text{for all } t\in [0,T]\text{ and }x\in\R,
\end{equation*}
it satisfies Assumption~\ref{cvallois} in view of Lemma~\ref{lem:monotone} below.  
Then the process $(B_{\t_t})_{0\le t\le T}$ is a fake Brownian motion, i.e., a Markov martingale with the same marginal distributions as a Brownian motion that is not a Brownian motion. 

\begin{lemma}\label{lem:monotone}
 If $(\m_t)_{0\le t\le T}$ is a continuous Gaussian peacock, then the map $t\mapsto \psi_t(x)$ is decreasing for all $x>0$.
\end{lemma}

\begin{proof}
 For the sake of clarity, we denote $R_t(x):=\mu_t([x,\infty[)$ in this proof. By integration by parts, it holds for all $x\geq 0$,
\begin{equation*}
 \ps_t(x) = \int_0^x {\frac{y\mu_t(y)}{R_t(y)} dy} = \int_0^x {\log{\left(\frac{R_t\left(y\right)}{R_t\left(x\right)}\right)} dy}.
\end{equation*}
Further, by change of variable, we have $R_t(x)= R_1(\frac{x}{\sqrt{t}})$. To conclude, it is clearly enough to prove that the map $t\mapsto \frac{R_1(ty)}{R_1(tx)}$ is increasing for all $x>0$ and $0<y<x$. By direct differentiation, we see that its derivative has the same sign as
\begin{multline*}
 xe^{-\frac{t^2x^2}{2}}\int_{ty}^{\infty}e^{-\frac{z^2}{2}}dz - ye^{-\frac{t^2y^2}{2}}\int_{tx}^{\infty}e^{-\frac{z^2}{2}}dz \\
 = \int_{txy}^{\infty}\Big(e^{-\frac{1}{2}\big(\frac{z^2}{x^2}+t^2x^2\big)}- e^{-\frac{1}{2}\big(\frac{z^2}{y^2}+t^2y^2\big)}\Big)dz.
\end{multline*}
It remains to observe that the quantity above is positive since 
\begin{equation*}
 \frac{z^2}{x^2}+t^2x^2 < \frac{z^2}{y^2}+t^2y^2,\quad \text{for all }z> txy,\, 0< y< x \text{ and }x> 0. \hfill\quad \qed
\end{equation*}
\end{proof}

\section{Conclusions}

This paper makes contribution on several topics related to the Vallois embedding and its applications. In particular, we provide a complete study of the robust hedging problem for options on local time in the one-marginal case by using the stochastic control approach.  In addition, we derive a new solution to the two-marginal Skorokhod embedding when the marginal distributions are symmetric. Under appropriate assumptions, we compute the corresponding monotone embedding functions. 
A natural direction for future research is to relax the monotonicity assumption in order to embed more marginals.
 Besides it would be of interest to iterate the stochastic control approach to deal with the multi-marginal robust hedging problem in the spirit of Henry-Labord\`ere et al.~\cite{henry-al-15}. However the problem is less tractable than one might hope and we have not yet been able to provide a complete solution.

\begin{acknowledgements}
	Julien Claisse is grateful to the financial support of ERC Advanced Grant 321111 ROFIRM.
\end{acknowledgements}

\bibliographystyle{ieeetr}
\bibliography{Vallois}{}

\end{document}